\documentclass[11pt]{amsart}

\usepackage{amssymb,amsthm}
\usepackage{amsmath}

\newtheorem{theorem}{Theorem}[section]

\newtheorem{lemma}[theorem]{Lemma}

\def\irr#1{{\rm  Irr}(#1)}

\def\cd#1{{\rm  cd}(#1)}

\def\phi{{\varphi}}

\begin{document}
\title[Derived length $4$ and $4$ character degrees]{Constructing solvable groups with derived length four and four character degrees}
\author[Mark L. Lewis]{Mark L. Lewis}

\address{Department of Mathematical Sciences, Kent State University, Kent, OH 44242}
\email{lewis@math.kent.edu}


\begin{abstract}
In this paper, we present a new method to construct solvable groups with derived length four and four character degrees.  We then use this method to present a number of new families of groups with derived length four and four character degrees.
\end{abstract}

\maketitle

\section{Introduction}

Throughout this paper, all groups will be finite and solvable.  The Taketa problem or alternatively, the Isaacs-Seitz conjecture, conjectures when $G$ is solvable that the derived length of $G$ is less than or equal to the number of character degrees.  In general, this conjecture is still open, but it was settled when $G$ has four character degrees by Garrison in his dissertation \cite{Garrison}.  In this paper, we are interested in solvable groups with exactly four character degrees and derived length four.

In Section 2 of \cite{4dlcd}, we listed all of the solvable groups with four character degrees and derived length four that we knew at that time.  So far as we know, no other examples have been published since that time.  All of the examples in Section 2 of \cite{4dlcd} have Fitting height $3$.  Also, none of the examples in that list have a nonabelian normal Sylow $p$-subgroup for some prime $p$.  And every example in that list has at least one character degree that is a prime.

In this paper, we provide a way to construct a group $G$ that has a nonabelian normal Sylow $p$-subgroup and has ${\rm dl} (G) = 4$ and $|\cd G| = 4$.  (We will use ${\rm dl} (G)$ to denote the derived length of $G$ and $\cd G = \{ \chi (1) \mid \chi \in \irr G \}$ for the set of character degrees of $G$.)  We will use this construction to produce the following:

\begin{theorem}\label{fitting two}
There exists a group $G$ with ${\rm dl} (G) = 4$, $|\cd G| = 4$, and Fitting height $2$.
\end{theorem}

We will also provide examples where no character degree is a prime (see Theorem \ref{no prime}).

In \cite{4dlcd}, we suggested that it might be possible to classify the solvable groups with four character degrees and derived length four.  The construction in this paper provides a counter position.  In particular, we will use our construction to produce seven different ``families'' of solvable groups with derived length four and four character degrees.  Since Section 2 of \cite{4dlcd} had seven families, one could view that we have doubled the number of families of examples.  On the other, we have in no way exhausted the groups that can produced from our construction.  In particular, all of the families we produce have that the normal Sylow $p$-subgroup is either an extraspecial group or is a Heisenberg group (i.e., a Sylow $p$-subgroup of ${\rm GL}_3 (p^a)$ for some positive integer $a$).  On the other hand, the normal Sylow $p$-subgroup found in the construction is more general, and we would expect that there will be many examples that do not have $P$ as one of these two groups.  We decided that seven families was more than sufficient to make our point.

We would like to thank Ni Du and Thomas Keller for several helpful conversations as we were writing this paper.

\section{Key Lemma}

In this note, we produce new solvable groups having derived length $4$ and four character degrees.

\begin{lemma} \label{facts}
Let $p$ be a prime, and let $P$ be a $p$-group so that $\cd P = \{ 1, p^\alpha \}$ for some positive integer $\alpha$.  Suppose the group $H$ acts via automorphisms on $P$ and that $H$ satisfies the following hypotheses: $p$ does not divide $|H|$ and $\cd H = \{ 1, a \}$ for some positive integer $a$. Let $C = C_H (P')$ and $D = C_P (C)$.  Assume one of the following:
\begin{enumerate}
\item $C = 1$ and $H$ acts Frobeniusly on $P$.
\item $C > 1$ is abelian, $D < P$, $H$ acts Frobeniusly on $P/D$, every nonlinear character in $\irr P$ is fully ramified with respect to $P/D$, and $H/C$ acts Frobeniusly on $D$.  When $P' < D$, assume that $|H:C| = a$.
\end{enumerate}
If $G = P \rtimes H$, then ${\rm dl} (G) = 4$ and $\cd G = \{ 1, a, |H|, |H:C| p^\alpha \}$.
\end{lemma}

\begin{proof}
Suppose first that $H$ acts Frobeniusly on $P$.  Since $H$ is nonabelian, we have $H' > 1$, so $H'$ acts Frobeniusly on $P$.  It follows that $P = [P,H'] \le G'$, so $P = [P,H'] \le [G',G'] = G''$.  On the other hand, since $|\cd {G/P}| = 2$, we have $G'' \le P$; so $G'' = P$.  Because $P'' = 1$, we conclude that ${\rm dl} (G) = 4$ in this case.

Suppose (2).  Since $C$ centralizes $P'$, it follows that $P' \le D$.  This implies that $D$ is normal in $P$.  Also, because $C$ is normal in $H$, we see that $D$ is normalized by $H$, so $D$ is normal in $G$.  Observe that $C$ abelian and $H$ nonabelian implies that $C < H$.  Since $H/C$ acts Frobeniusly on $D$, so $D = [D,H] \le G'$.  The fact that $H$ acts Frobeniusly on $P/D$ implies that $P/D = [P/D,H] = [P,H]D/D \le (G/D)' = G'D/D = G'/D$.  This implies that $P \le G'$.  We now have that $G' = PH'$.  We have $H' > 1$ and $H$ acts Frobeniusly on $P/D$, so $H'$ acts Frobeniusly on $P/D$.  We have two cases to deal with, when $H' \le C$ and when $H' \not\le C$.

Suppose that $H' \le C$.  It follows that $H'$ centralizes $D$.  Since every nonlinear irreducible character of $P$ is fully-ramified with respect to $P/D$, this implies that $H'$ fixes every nonlinear irreducible character of $P$.  Applying Theorem 3.3 of \cite{coprime}, we have $[P,H']' = P'$.  Observe that $[P,H'] \le G'' \le P$.  This implies that $[P,H']' \le G''' \le P'$.  We deduce that $G''' = P'$, and since $P' > 1$ is abelian, we conclude that ${\rm dl} (G) = 4$.

Now, we consider the case where $H' \not\le C$.  In particular, $H' \cap C < H'$.  It follows that $H'/H' \cap C$ acts Frobeniusly on $D$.  Hence, we have $D = [D,H'] \le G'$.  Since $H$ acts Frobeniusly on $P/D$ and $H' > 1$, it follows that $H'$ acts Frobeniusly on $P/D$.  This implies that $P/D = [P/D,H']$, and so, $P = [P,H']D = [P,H'][D,H'] = [P,H']$.  Observe that $[P,H'] \le G'$; thus, $P \le G'$.  We then obtain $P = [P,H'] \le [G',G'] = G'' \le P$, and hence, $G'' = P$.  Because $P' > 1 $ and $P'' = 1$, we conclude that ${\rm dl} (G) = 4$.

We now compute $\cd G$.  We have $\cd {G/P} = \cd H = \{ 1, a \}$.  If $H$ acts Frobeniusly on $P$, then $\theta^G \in \irr G$ for all $\theta \in \irr P$ (see Theorem 6.34 (b) of \cite{text}), and so $\cd G = \{ 1, a, |H|, |H| p^\alpha \}$ which is the desired result since $C = 1$.  Thus, we assume we have hypothesis (2).

Since $H$ acts Frobeniusly on $P/D$, we have if $1 \ne \lambda \in \irr {P/D}$, then $\lambda^G \in \irr G$, and so, $|H| = \lambda^G (1) \in \cd G$.  We deduce that $\cd {G/D} = \{ 1, a, |H| \}$.  When $D = P'$, we have $\cd {G/P'} = \cd {G/D}$.

Suppose that $P' < D$.  Consider $\delta \in \irr {D/P'}$.  Observe that $\delta$ is $C$-invariant, and since $H/C$ acts Frobeniusly on $D$ and thus on $D/P'$, we see that $C$ is the stabilizer of $\delta$ in $H$.  We know that $\delta$ extends to $\irr {P/P'}$.  Note that $C$ acts on the extensions of $\delta$ to $P$ and using Gallagher's theorem (Corollary 6.17 of \cite{text}) $\irr {P/D}$ acts transitively by right multiplication on the extensions of $\delta$ to $P$.  Applying Glauberman's lemma (Lemma 13.8 of \cite{text}), we see that $\delta$ has a $C$-invariant extension $\mu \in \irr {P/P'}$.  Since $C$ acts Frobeniusly on $\irr {P/D}$, we may apply Corollary 13.9 of \cite{text} to see that $\mu$ is the unique $G$-invariant extension of $\delta$ to $\mu$.  Note that $C$ will be the stabilizer of $\mu$ in $H$, and so, $PC$ is the stabilizer of $\mu$ in $G$.  Applying Corollary 6.27 of \cite{text}, $\mu$ extends to $PC$ and by Gallagher's theorem, $\mu$ only has extensions to $PC$ since $C$ is abelian.  This implies that $\cd {G \mid \mu} = \{ |G:PC| \} = \{ |H:C| \} = \{ a \}$ since $a = |H:C|$ in this case.

Note that any extension of $\delta$ to $P$ will have the form $\mu \lambda$ for some character $\lambda \in \irr {P/D}$.  If $1 \ne \lambda$, we see that if $h \in C_H (\mu\lambda)$ then $h$ stabilizes $(\mu\lambda)_D = \delta$, and so, $C_H (\mu \lambda) \le C_H (\delta) = C = C_H (\mu)$, and so, $\mu \lambda = (\mu \lambda)^h = \mu \lambda^h$.  Applying Gallagher's theorem, we have that $h$ stabilizes $\lambda$.  Since $C$ acts Frobeniusly on $\irr {P/D}$ and $\lambda \ne 1$, we conclude that $h =1$.  It follows that $C_H (\mu\lambda) = 1$, and so, $(\mu\lambda)^G \in \irr G$.  We deduce that $\cd {G \mid \delta} = \{ a, |H| \}$.  This yields $\cd {G/P'} = \{ 1, a, |H| \} = \cd {G/D}$ in this case.

Finally, suppose $1 \ne \gamma \in \irr {P'}$, and consider $\hat\gamma \in \irr {D \mid \gamma}$.  Notice that the irreducible constituents of $\hat\gamma^P$ will be nonlinear, so $\hat\gamma$ is fully-ramified with respect to $P/D$, and let $\epsilon \in \irr P$ be the unique irreducible constituent of $(\hat\gamma)^P$.  We know that $\epsilon (1) = p^\alpha$.  It follows that $\epsilon$ and $\hat\gamma$ have the same stabilizer in $G$.  Because $C$ centralizes $D$ and $H/C$ acts Frobeniusly on $D$, it follows that $PC$ will be the stabilizer of $\hat\gamma$ in $G$.  Applying Corollary 6.28 of \cite{text}, we see that $\epsilon$ extends to $PC$, and since $C$ is abelian, we may use Gallagher's theorem to see that $\epsilon$ only has extensions to $PC$.  We obtain $\cd {G \mid \hat\gamma} = \{ |G:PC| \epsilon (1) \} = \{ |H:C| p^\alpha \}$.  We conclude that $\cd G = \{ 1, a, |H|, |H:C| p^\alpha \}$ as desired.
\end{proof}

\section{Specific Families}

We now find specific families of groups that meet the parameters of Lemma \ref{facts}.  We begin with a family of groups based on the Heisenberg group.  In this first example, $G$ will be a Frobenius group where the Frobenius kernel is a Heisenberg group and a Frobenius complement is a nonnilpotent metacyclic group.

\begin{theorem} \label{one}
Let $p$ be a prime and let $q$ be an odd prime that divides $p-1$.  Then there exists a group $G$ with ${\rm dl} (G) = 4$, Fitting height $3$, and $\cd G = \{ 1, q, (p^q-1)/(p-1))_{q'} (p-1)_q q, p^q ((p^q-1)/(p-1))_{q'} (p-1)_q q \}$.
\end{theorem}

\begin{proof}
We are going to take $P$ to be the Heisenberg group of order $p^{3q}$.  We represent $P$ as follows.  Let $F$ be the field of order $p^q$.  Then we can view $P$ as
$$
\left\{ \left[\begin{array}{ccc} 1 & a & c \\ 0 & 1 & b \\ 0 & 0 & 1 \end{array}\right] \mid a,b,c \in F \right\}.
$$
We write $F^*$ for the multiplicative group of $F$ and $\mathcal {G}$ for the Galois Group of $F$ with respect to $Z_p$.  It is easy to see that $\mathcal {G}$ acts on $F^*$ and that the resulting semi-direct product $F^* \mathcal {G}$ is isomorphic to the semi-linear group $\Gamma (F)$.  (See page 37 of \cite{MaWo} for the definition of the semilinear group.)  We can define an action by automorphisms for $\Gamma (F)$ on $P$ as follows: if $\lambda \in F^*$, then
$$
\left[ \begin{array}{ccc} 1 & a & c \\ 0 & 1 & b \\ 0 & 0 & 1  \end{array} \right] \cdot \lambda = \left[ \begin{array}{ccc} 1 & \lambda a & \lambda^2 c \\ 0 & 1 & \lambda b \\ 0 & 0 & 1 \end{array} \right]
$$
where the multiplication is in $F$ and if $\sigma \in \mathcal {G}$ then
$$
\left[ \begin{array}{ccc} 1 & a & c \\ 0 & 1 & b \\ 0 & 0 & 1  \end{array} \right] \cdot \sigma = \left[ \begin{array}{ccc} 1 & a^\sigma & c^\sigma \\ 0 & 1 & b^\sigma \\ 0 & 0 & 1 \end{array} \right].
$$
Notice that $\lambda$ has odd order, then the action of $\lambda$ on $P$ is Frobenius.

Now, let $\gamma$ be an element of $F^*$ of order $((p^q-1)/(p-1))_{q'}$.  Notice that the order of $\gamma$ is odd, so $\gamma^2 \ne 1$.  Let $\lambda$ be a generator for the Sylow $q$-subgroup of $F^*$ and let $\sigma \in \mathcal {G}$ be the Frobenius automorphism, so $\sigma$ has order $q$.  We define $H = \langle \gamma, \lambda \sigma \rangle$.  It is not difficult to see that $\sigma$ does not commute with $\gamma$ and so $H$ is not abelian.  On the other hand, $H$ has a normal abelian subgroup of index $q$, so $\cd H = \{ 1, q \}$.  We observe that
$$
(\lambda \sigma)^q = \lambda^{\sigma^q} \dots \lambda^{\sigma^2} \lambda^\sigma = \lambda^{p^{q-1}} \cdots  \lambda^p \lambda = \lambda^{1 + p + \dots +p^{q-1}} = \lambda^{(p^q-1)/(p-1)}.
$$
Since $q$ divides $p-1$, it follows that $\lambda^{(p^q-1)/(p-1)} \ne 1$.  Note that $Q = \langle \lambda\sigma \rangle$ is a Sylow $q$-subgroup of $H$, is cyclic, has order $(p^q - 1)q$, and acts Frobeniusly on $P$.  Also, observe that $\langle \gamma \rangle$ is a Hall $q$-complement of $H$, has order $((p^q-1)/(p-1))_{q'}$, and also acts Frobeniusly on $P$.  Observe that $C_H (P') = 1$ and $|H| = ((p^q-1)/(p-1))_{q'} (p-1)_q q$.  We conclude that $H$ and $P$ satisfy the hypotheses of Lemma \ref{facts}, and we obtain the conclusion from there.  Note that $P$ is the Fitting subgroup of $G$ and $H$ has Fitting height $2$, so $G$ has Fitting height $3$.
\end{proof}

In this next example, the normal Sylow $p$-subgroup is again a Heisenberg group.  In this case, $G$ is not a Frobenius group.

\begin{theorem} \label{two}
Let $p$ be a prime, let $q$ be an odd prime that divides $p-1$, and let $r$ be a divisor of $p-1$ that is coprime to $q$.  Then there exists a group $G$ with ${\rm dl} (G) = 4$, Fitting height $3$, and $\cd G = \{ 1, q, ((p^q-1)/(p-1))_{q'} (p-1)_q q r, p^q ((p^q-1)/(p-1))_{q'} (p-1)_q q \}$.
\end{theorem}

\begin{proof}
We again take $P$ to be the Heisenberg group of order $p^{3q}$ and we take $K$ to the subgroup $H$ from Theorem \ref{one}, so $|K| = ((p^q-1)/(p-1))_{q'} (p-1)_q q$.  Let $\eta$ be an element of $F^*$ with order $r$.  It is known that the gcd of $(p^q-1)/(p-1)$ and $p-1$ is $q$ (see \cite{nor}).  Thus, since $r$ divides $p-1$ and is coprime to $q$, it follows that $r$ is coprime to $((p^q-1)/(p-1))_{q'}$.  We can define an action of $\eta$ on $P$ by
$$
\left[ \begin{array}{ccc} 1 & a & c \\ 0 & 1 & b \\ 0 & 0 & 1  \end{array} \right] \cdot \eta = \left[ \begin{array}{ccc} 1 & \eta a & c \\ 0 & 1 & \eta^{-1} b \\ 0 & 0 & 1 \end{array} \right]
$$
where the multiplication is in $F$.  Note that the action of $\eta$ will commute with the actions of $\lambda$ and $\gamma$.  Also, notice that $\eta$ will be in the prime subfield of $F$, so $\eta$ and $\sigma$ commute.  We deduce that $\eta$ centralizes $K$.  Thus, we can take $H = K \times \langle \eta \rangle$.  Notice that $\eta$ acts Frobeniusly on $P/P'$ and centralizes $P'$.  It follows that $C = C_H (P') = \langle \eta \rangle$ and $D = C_P (C) = P'$.  In particular, $G = P \rtimes H$ satisfies the hypotheses of Lemma \ref{facts}, and we obtain the derived length and character degrees from that lemma.  It is easy to see that $G$ has Fitting height $3$.
\end{proof}

This next two groups are based on extraspecial groups of order $p^{2q+1}$ where $q$ is a prime.  In the next theorem, we have that $|H:C| = q$ using the notation of Lemma \ref{facts}.

\begin{theorem} \label{three}
Let $p$ be a prime and let $q$ be an odd prime that divides $p-1$.  Then there exists a group $G$ with ${\rm dl} (G) = 4$, Fitting height $3$, and $\cd G = \{ 1, q, (p^q-1)_{q'} (p-1)_q q, p^q q \}$.
\end{theorem}

\begin{proof}
Let $V$ be a vector space of dimension $q$ over $Z_p$, the field of order $p$.  Let $\hat V$ be the dual space for $V$; that is, $\hat V$ is the set of all linear transformations from $V$ to $Z_p$.  We define $P = \{ (a,\alpha,z) \mid a \in V, \alpha \in \hat V, z \in Z_p \}$ where multiplication in $P$ is defined by $(a_1,\alpha_1,z_1)(a_2,\alpha_2,z_2) = (a_1 + a_2, \alpha_1 + \alpha_2, z_1 + z_2 + \alpha_2 (a_1))$.  It is not difficult to see that $P$ is an extraspecial $p$-group of order $p^{2q +1}$.

If $\delta$ is an automorphism of $V$, then we obtain an automorphism for $\hat V$ by defining $\alpha^\delta$ by $\alpha^\delta (v) = \alpha ( v^{\delta^{-1}})$ for all $v \in V$.  Note that $\alpha^\delta (v^\delta) = \alpha (v)$.  It is not difficult to see that we can define an automorphism on $P$ by $(v,\alpha,z)^\delta = (v^\delta,\alpha^\delta,z)$.

We can identify $V$ with the additive group of the field $F$ of order $p^q$.  If $\lambda$ is a nonzero element of $F$, then multiplication by $\lambda$ yields an automorphism of $V$, and we use $v^\lambda$ to denote this map on $V$ and $\alpha^\lambda$ to denote the associated map on $\hat V$.  Also, the Galois automorphisms of $F$ will yield automorphisms of $V$.  If $\sigma$ is a Galois automorphism of $F$, then we use $v^\sigma$ to be the automorphism on $V$ and $\alpha^\sigma$ for the associated map on $\hat V$.  We take $\gamma$ to be a generator of the Hall $q$-complement of $F^*$.  We take $\lambda$ to be a generator for the Sylow $q$-subgroup of $F^*$.  We take $\sigma$ to be the Frobenius automorphism.  We use these same letters to denote the automorphisms of $P$ given by each of these elements as above.

Let $x$ be a non-zero element of $Z_p$, then we can view $x$ as element of $F^*$, and so we can define the action of $x$ on $V$ and $\hat V$ as before.  Notice that since the elements of $\hat V$ are linear transformations, we have that $\alpha (v^x) = \alpha (xv) = x \alpha (v)$ and $\alpha^{x^{-1}} (v) = \alpha (v^x)$, so $\alpha^{x^{-1}} (v^x) = x^2 \alpha (v)$.  Thus, we define an automorphism of $P$ by $(v,\alpha,z) \mapsto (v^x,\alpha^{x^{-1}},x^2 z)$.  Let $\xi$ be this automorphism defined for an element $x$ of order $q$ in $Z_p^*$.  It is easy to see that $\xi$ will commute with $\gamma$ and $\lambda$ as automorphisms of $P$.  Since as a Galois automorphism, $\sigma$ fixes the elements in $Z_p$, it is not difficult to see that $\sigma$ and $\xi$ will commute.  Let $H = \langle \gamma, \lambda\xi\sigma  \rangle$.  Since $\gamma$ commutes with $\lambda$ and $\xi$, but not $\sigma$, we see that $H$ is not abelian.  Working as in the proof of the last theorem, we see that $(\lambda \xi \sigma)^q = \lambda^{(p^q-1)/(p-1)}$ which will have order $(p-1)_q$.  Notice that $C = \langle \gamma, \lambda^{(p^q-1)/(p-1)} \rangle$ has index $q$ in $H$, is abelian, and centralizes $P'$.   It follows that $\cd H = \{ 1, q \}$ and $C = C_H (P')$.  Observe that $P' = C_P (C)$ which is $D$ in the notation of Lemma \ref{facts}.  It is not difficult to see that $H$ acts Frobeniusly on $P/D$ and that $H/C$ acts Frobeniusly on $D$, so the hypotheses of Lemma \ref{facts} are met.  We obtain the derived length and the character degree set conclusions from that result. Notice that $P$ is the Fitting subgroup of $G$ and $G/P \cong H$ has Fitting height $2$, so $G$ has Fitting height $3$.
\end{proof}

In this next example, we again have an extraspecial group of order $p^{2q+1}$, but in this case, we have $|H:C|$ is relatively prime to $q$, again using the notation of Lemma \ref{facts}.

\begin{theorem} \label{four}
Let $p$ be a prime, let $q$ a prime that divides $p-1$, and let $r$ be an odd divisor of $p-1$ that is relatively prime to $q$.  Then there exists a group $G$ with ${\rm dl} (G) = 4$, Fitting height $3$, and $\cd G = \{ 1, q, (p^q-1)_{\{q,r\}'} (p-1)_q q r, p^q r \}$.
\end{theorem}

\begin{proof}
As in the proof of Theorem \ref{three}, we take $V$ to a vector space of dimension $q$ over $Z_p$, we write $\hat V$ for the dual space for $V$, and $P$ for the associated extraspecial group.  Again, we take $F$ to be the field of order $p^q$ and we have the same action for elements of $F^*$ and the Galois group of $F$ on $P$.  We take $\gamma$ to be a generator for the Hall $\{q,r\}$-complement of $F^*$, $\lambda$ to be a generator for the Sylow $q$-subgroup of $F^*$, and $\sigma$ to be the Frobenius automorphism of $F$.  We now take $x$ to be an element of order $r$ in $Z_p^*$ and we let $\xi$ be the automorphism of $P$ defined for $x$.  Take $H = \langle \gamma, \lambda \sigma, \xi \rangle$.  Since $\gamma$ commutes with $\lambda$, but not $\sigma$, we see that $H$ is not abelian.  We see that $(\lambda \sigma)^q = \lambda^{(p^q-1)/(p-1)}$ which will have order $(p-1)_q$.  Observe that $\langle \gamma, \lambda^{(p^q-1)/(p-1)}, \xi \rangle$ is a normal, abelian subgroup of index $q$, so $\cd H = \{ 1, q \}$.  Also, $C = C_H (P') = \langle \gamma, \lambda \sigma \rangle$ and $D = C_P (C) = P'$.  It is not difficult to see that $H$ acts Frobeniusly on $P/D$ and that $H/C$ acts Frobeniusly on $D$, so the hypotheses of Lemma \ref{facts} are met.  We obtain the derived length and the character degree set conclusions from that result. Notice that $P$ is the Fitting subgroup of $G$ and $G/P \cong H$ has Fitting height $2$, so $G$ has Fitting height $3$.
\end{proof}

We now present an example where $P' < D = C_P (C_H(P'))$.  A careful reading of the previous examples will show that $P' = D = C_P (C_H (P'))$ in all of them.

\begin{theorem} \label{five}
Let $p$ be a prime, let $q$ an odd prime that divides $p-1$, and let $n > q$ be an integer.  Then there exists a group $G$ with ${\rm dl} (G) = 4$, Fitting height $3$, and $\cd G = \{ 1, q, (p^q-1)_{q'} (p-1)_q q, p^n q \}$.
\end{theorem}

\begin{proof}
Let $P_1$ be the group $P$ from the Theorem \ref{three}.  We take $P_2$ to be an extraspecial group of order $p^{2(n-q)}$ and exponent $p$.   We will take $P$ to be a central product of $P_1$ and $P_2$, and we let $H$ be as in Theorem \ref{three}.  We have $H$ act on $P_1$ as it acted on $P$ in Theorem \ref{three}.  We will have $\gamma$, $\lambda$ and $\sigma$ act trivially on $P_2$, and it is not difficult to see that there is a Frobenius action of $x$ on $P_2$ so that the action on $Z(P_2)$ matches the action of $x$ on $Z(P_1)$.  This then defines an action of $H$ on $P$.  Notice that $C = C_H (P') = \langle \gamma, \lambda^{(p^q-1)/(p-1)} \rangle$ and $D = C_P (C) = P_2$ so $P' < D$.  Observe that $|H:C| = q$.  Also, all of the nonlinear irreducible characters of $P$ are fully-ramified with respect to $D$, so that the hypotheses of Lemma \ref{facts} are met. We obtain the conclusions regarding the derived length and character degrees from there.  The Fitting height follows as in Theorem \ref{three}.
\end{proof}

Next we present an example where $n$ is not a prime where $P$ is the normal Sylow $p$-subgroup and $\cd P = \{ 1, p^n \}$.  Note that in the previous examples, we have had $n$ as a prime.  Recall that $q$ is a Zsigmondy prime divisor of $p^n - 1$ for positive integers $p$ and $n$ if $q$ divides $p^n - 1$ and $q$ does not divide $p^a - 1$ for integers $a$ such that $1 \le a < n$.  Observe that none of the character degrees in this example is a prime.

\begin{theorem} \label{no prime}
Let $p$ be a prime and let $n$ be an odd integer so that every prime divisor of $n$ divides $p-1$.  Let $\pi$ be the set of prime divisors of $n$, let $\rho$ be the set of Zsigmondy prime divisors of $p^n - 1$, and let $m$ be an integer so that $m$ divides $n (p-1)_\rho$, $n$ divides $m$, and every prime divisor of $(p-1)_\rho$ divides $m/n$.  Then there exists a group $G$ with ${\rm dl} (G) = 4$, Fitting height $3$, and $\cd G = \{ 1, n, (p^n-1)_\rho m, p^n (p^n-1)_\rho m \}$.
\end{theorem}

\begin{proof}[Sketch of proof]
We take $P$ to be the Heisenberg group of order $p^{3n}$ and let $F$ be the field of order $p^n$.  Working as in the proof of Theorem \ref{one}, we can define an action of $F^*$ and ${\rm Gal} (F)$ on $P$.  Applying Theorem 11 of \cite{LeRi}, we can find subgroups $K$ and $N$ of $F^* {\rm Gal} (F)$ so that $|K| = (p^n-1)_\rho$, $|N| = m$, $\cd {NK} = \{ 1, n\}$, $K$ is cyclic, $N$ is nilpotent, and $NK$ acts Frobeniusly on $P$.  Take $G = PNK$. We now apply Lemma \ref{facts} to obtain the conclusion.
\end{proof}

We now produce an example with Fitting height $2$.  Notice that this yields Theorem \ref{fitting two}.

\begin{theorem}
Let $p$ be a prime that is congruent to $3$ modulo $8$.  Then there exists a group $G$ with ${\rm dl} (G) = 4$, Fitting height $2$, and $\cd G = \{ 1, 2, 8, 2 p^2 \}$.
\end{theorem}

\begin{proof}
We somewhat follow the construction found in the proof of Theorem \ref{three}.  We take $V$ to be a vector space of dimension $2$ over $Z_p$, and we define $P$ as in the proof of Theorem \ref{three} to be the extraspecial group of order $p^5$ arising from paring $V$ with $\hat V$.  Viewing $V$ as a field of order $p^2$, it is not difficult to see that the multiplicative group will have an element of order $8$.  Let $\lambda$ be the automorphism of $V$ that is obtained by multiplication from that element, and as in the proof of Theorem \ref{three}, $\lambda$ also determines an automorphism of $\hat V$ of order $8$, and we also use $\lambda$ to denote the automorphism of $P$ given by $(a,\alpha,z) \mapsto (a^\lambda,\alpha^\lambda,z)$.  We let $\sigma$ be the Frobenius automorphism for $V$ viewed as field, and again, $\sigma$ defines an automorphism of $\hat V$, and we write $\sigma$ for the automorphism of $P$ given by $(a,\alpha,z) \mapsto (a^\sigma,\alpha^\sigma,z)$.  It is not difficult to see that $(\lambda \sigma)^2 = \lambda^{p+1}$, and since $p \equiv 3 (~{\rm modulo}~8)$, we see that $\lambda^{p+1} = \lambda^4 = -1$.

Let $\zeta$ be an element of order $2$ in the multiplicative group of $Z_p$, and observe that the map $(a,\alpha,z) \mapsto (\zeta a, \alpha, \zeta z)$ is an automorphism of order $2$ on $P$ and will centralize $\lambda$ and $\sigma$ as automorphisms of $P$.  We now take $H$ to be the subgroup of the automorphism group of $P$ given by $\langle \zeta \lambda^2, \lambda \sigma \rangle$.  It is not difficult to see that $H$ will be isomorphic to the quaternion group of order $8$.  Since $\lambda^4 = -1$, we see that $H$ acts Frobeniusly on $P/P'$.  Observe that $\lambda \sigma$ centralizes $P'$, and $\zeta \lambda^2$ does not centralize $P'$.  Since $|H:\langle \lambda\sigma \rangle| = 2$, we conclude that $C = \langle \lambda\sigma \rangle$.  Notice that $H/C$ acts Frobeniusly on $P'$.  Thus, the hypotheses of Lemma \ref{facts} are met, and we obtain that ${\rm dl} (G) = 4$ and the character degrees are as stated.  Since $H$ is nilpotent, we see that $G$ has Fitting height $2$.
\end{proof}


\begin{thebibliography}{99}
\bibitem{4dlcd} N.~Du and M.~L.~Lewis, Groups with four character degrees and derived length four.  {\it Comm. Algebra} {\bf 43} (2015), 4660-4673.

\bibitem{Garrison} S.~Garrison, On groups with a small number of character degrees, PhD thesis, University of Wisconsin, Madison, 1973.

\bibitem{hup}  B.~Huppert, ``Endliche Gruppen I,'' Springer-Verlag, Berlin, 1983.

\bibitem{brown} I.~M.~Isaacs, Characters of solvable and symplectic groups. {\it Amer. J. Math.} {\bf 95} (1973), 594-635.

\bibitem{text}  I.~M.~Isaacs, ``Character Theory of Finite Groups,'' Academic Press, San Diego, California, 1976.

\bibitem{coprime} I.~M.~Isaacs, Coprime group actions fixing all nonlinear irreducible characters. {\it Canad. J. Math.} {\bf 41} (1989), 68-82.

\bibitem{struct} M.~L.~Lewis, Determining group structure from sets of irreducible character degrees. {\it J. Algebra} {\bf 206} (1998), 235-260.

\bibitem{twocomp} M.~L.~Lewis, Solvable groups whose degree graphs have two connected components. {\it J. Group Theory} {\bf 4} (2001), 255-275.

\bibitem{LeRi} M.~L.~Lewis and J.~M.~Riedl, Affine semi-linear groups with three irreducible character degrees. {\it J. Algebra} {\bf 246} (2001), 708-720.

\bibitem{LeMo} M.~L.~Lewis and A.~Moret\'o, Bounding the number of irreducible character degrees of a finite group in terms of the largest degree. {\it J. Algebra Appl.} {\bf 13} (2014), 1350096, 18 pp.

\bibitem{LiWa} H.~Liu and Y.~Wang, The automorphism group of a generalized extraspecial p-group.  {\it Sci. China Math.} {\bf 53} (2010), 315-334.

\bibitem{MaWo} O.~Manz, T.~ R.~Wolf, ``Representations of Solvable Groups,'' Cambridge University Press, Cambridge, 1993.

\bibitem{nor}  T.~Noritzsch, Groups having three complex irreducible character degrees. {\it J. Algebra} {\bf 175} (1995), 767-798.

\end{thebibliography}
\end{document}